\documentclass[oneside,english]{amsart}
\usepackage[T1]{fontenc}
\usepackage[latin9]{inputenc}
\usepackage{mathrsfs}
\usepackage{amstext}
\usepackage{amsthm}
\usepackage{amssymb}
\usepackage[all]{xy}

\makeatletter
\numberwithin{equation}{section}
\numberwithin{figure}{section}
\theoremstyle{plain}
\newtheorem{thm}{\protect\theoremname}
\theoremstyle{plain}
\newtheorem{lem}[thm]{\protect\lemmaname}
\theoremstyle{plain}
\newtheorem{cor}[thm]{\protect\corollaryname}
\theoremstyle{plain}
\newtheorem{prop}[thm]{\protect\propositionname}
\theoremstyle{remark}
\newtheorem{rem}[thm]{\protect\remarkname}

\usepackage[dvipsnames,svgnames,x11names,hyperref]{xcolor}
\usepackage[colorlinks=true,linkcolor=Dark Red, citecolor=Dark Green,linktoc=all]{hyperref}
\usepackage{lmodern}
\renewcommand*{\epsilon}{\varepsilon}
\usepackage{amssymb}
\usepackage{amsfonts}
\usepackage{egothic}

\usepackage[T1]{fontenc}

\usepackage{url}
\usepackage{amsthm}
\usepackage{aliascnt}
\usepackage{lipsum}
\usepackage{mathrsfs}
\usepackage{esint}
\usepackage{url}
\usepackage{amsmath}
\usepackage{dsfont}
\usepackage{bbm}
\usepackage{pdflscape}
\usepackage{bbm}
\usepackage{bussproofs}

\usepackage{tikz}
\tikzset{node distance=2.5cm, auto}
\usetikzlibrary{positioning}

\newcommand{\myar}[2]{\ar^-{#1}[#2]}

\makeatletter
\def\matrixobject@{%
  \edef \next@{={\DirectionfromtheDirection@ }}%
  \expandafter \toks@ \next@ \plainxy@
  \let\xy@@ix@=\xyq@@toksix@
  \xyFN@ \OBJECT@}
\let\xy@entry@@norm=\entry@@norm
\def\entry@@norm@patched{%
  \let\object@=\matrixobject@
  \xy@entry@@norm }
\AtBeginDocument{\let\entry@@norm\entry@@norm@patched}
\makeatother

\newcommand{\twocong}[2][0.5]{\ar@{}[#2] \save ?(#1)*{\cong}\restore}
\newcommand{\twoeq}[2][0.5]{\ar@{}[#2] \save ?(#1)*{=}\restore}
\newcommand{\ltwocell}[3][0.5]{\ar@{}[#2] \ar@{=>}?(#1)+/r 0.15cm/;?(#1)+/l 0.15cm/^{#3}}
\newcommand{\rtwocell}[3][0.5]{\ar@{}[#2] \ar@{=>}?(#1)+/l 0.15cm/;?(#1)+/r 0.15cm/^{#3}}
\newcommand{\utwocell}[3][0.5]{\ar@{}[#2] \ar@{=>}?(#1)+/d  0.15cm/;?(#1)+/u 0.15cm/_{#3}}
\newcommand{\dtwocell}[3][0.5]{\ar@{}[#2] \ar@{=>}?(#1)+/u  0.15cm/;?(#1)+/d 0.15cm/^{#3}}
\newcommand{\ultwocell}[3][0.5]{\ar@{}[#2] \ar@{=>}?(#1)+/dr  0.15cm/;?(#1)+/ul 0.15cm/^{#3}}
\newcommand{\urtwocell}[3][0.5]{\ar@{}[#2] \ar@{=>}?(#1)+/dl  0.15cm/;?(#1)+/ur 0.15cm/^{#3}}
\newcommand{\dltwocell}[3][0.5]{\ar@{}[#2] \ar@{=>}?(#1)+/ur  0.15cm/;?(#1)+/dl 0.15cm/^{#3}}
\newcommand{\drtwocell}[3][0.5]{\ar@{}[#2] \ar@{=>}?(#1)+/ul  0.15cm/;?(#1)+/dr 0.15cm/^{#3}}

\makeatletter
\renewcommand{\tocsection}[3]{%
  \indentlabel{\@ifnotempty{#2}{\bfseries\ignorespaces#1 #2\quad}}\bfseries#3}
\renewcommand{\tocsubsection}[3]{%
  \indentlabel{\@ifnotempty{#2}{\ignorespaces#1 #2\quad}}#3}

\newcommand\@dotsep{4.5}
\def\@tocline#1#2#3#4#5#6#7{\relax
  \ifnum #1>\c@tocdepth 
  \else
    \par \addpenalty\@secpenalty\addvspace{#2}%
    \begingroup \hyphenpenalty\@M
    \@ifempty{#4}{%
      \@tempdima\csname r@tocindent\number#1\endcsname\relax
    }{%
      \@tempdima#4\relax
    }%
    \parindent\z@ \leftskip#3\relax \advance\leftskip\@tempdima\relax
    \rightskip\@pnumwidth plus1em \parfillskip-\@pnumwidth
    #5\leavevmode\hskip-\@tempdima{#6}\nobreak
    \leaders\hbox{$\m@th\mkern \@dotsep mu\hbox{.}\mkern \@dotsep mu$}\hfill
    \nobreak
    \hbox to\@pnumwidth{\@tocpagenum{\ifnum#1=1\bfseries\fi#7}}\par
    \nobreak
    \endgroup
  \fi}
\AtBeginDocument{%
\expandafter\renewcommand\csname r@tocindent0\endcsname{0pt}
}
\def\l@subsection{\@tocline{2}{0pt}{2.5pc}{5pc}{}}
\def\l@section{\@tocline{2}{3pt}{2.5pc}{5pc}{}}
\makeatother

\usepackage{lipsum}

\renewcommand{\phi}{\varphi}

\DeclareMathOperator{\el}{el}

\newcommand{\cat}[1]{\mathrm{\mathcal #1}}

\renewcommand{\mathscr}{\mathcal}
\usepackage{microtype, stmaryrd, url, lmodern, eucal}
\usepackage[shortlabels]{enumitem}
\usepackage{mathtools}

\renewcommand{\mathbf}{\cat}
\renewcommand{\mathsf}{}

\newtheorem{ax}{Axiom}

\numberwithin{thm}{subsection}

\makeatother

\usepackage{babel}
\providecommand{\corollaryname}{Corollary}
\providecommand{\lemmaname}{Lemma}
\providecommand{\propositionname}{Proposition}
\providecommand{\remarkname}{Remark}
\providecommand{\theoremname}{Theorem}

\begin{document}
\title[Bicategories of Spans as Generic Bicategories]{Bicategories of Spans as Generic Bicategories}
\begin{abstract}
In a bicategory of spans (an example of a ``generic bicategory'')
the factorization of a span $\left(s,t\right)$ as the span $\left(s,1\right)$
followed by $\left(1,t\right)$ satisfies a simple universal property
with respect to all factorizations in terms of the generic bicategory
structure. Here we show that this universal property can in fact be
used to \emph{characterize} bicategories of spans. 

This characterization of spans is very different from the others in
that it does not mention any adjointness conditions within the bicategory.
\end{abstract}

\author{Charles Walker}
\address{Department of Mathematics and Statistics, Masaryk University, Kotl{\'a}{\v r}sk{\'a}
2, Brno 61137, Czech Republic }
\email{\tt{walker@math.muni.cz}}
\thanks{This work was supported by the Operational Programme Research, Development
and Education Project \textquotedblleft Postdoc@MUNI\textquotedblright{}
(No. CZ.02.2.69/0.0/0.0/16\_027/0008360)}

\maketitle
\tableofcontents{}

\section{Introduction\label{intro}}

Bicategories of spans were introduced by Bénabou \cite{ben1967} and
have since become one of the fundamental constructions in 2-dimensional
category theory. Given a category $\mathcal{E}$ with pullbacks, its
bicategory of spans, denoted $\cat{Span}\left(\mathcal{E}\right)$,
contains the same objects as those of $\mathcal{E}$, has 1-cells
given as diagrams
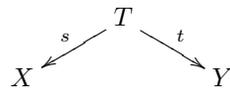
\[
\xymatrix@R=1em{ & T\ar[rd]^{t}\ar[ld]_{s}\\
X &  & Y
}
\]
called spans, 2-cells between such spans given as commuting diagrams
\[
\xymatrix@R=1em{ & T\ar[rd]^{t}\ar[ld]_{s}\ar[dd]|-{h}\\
X &  & Y\\
 & s\ar[ru]_{q}\ar[lu]^{p}
}
\]
and composition of 1-cells is defined by pullback, inducing a natural
definition of horizontal composition of 2-cells. Whilst an explicit
description as above is useful, one can gain a better understanding
of where these bicategories of spans come from by giving more abstract
characterizations.

The currently known characterizations all rely heavily on the adjoint
properties of spans, and the Beck conditions that these adjoints satisfy.
For example the universal property of spans (which may be viewed as
a characterization) says giving a pseudofunctor $\cat{Span}\left(\mathcal{E}\right)\to\mathscr{C}$
is equivalent to giving a pseudofunctor $\mathcal{E}\to\mathscr{C}$
which sends morphisms to left adjoints and satisfies the Beck condition
\cite{unispans}.

There is also another more interesting characterization of bicategories
of spans due to Lack, Walters, and Wood \cite{spanscartesian}. They
show that bicategories of spans are precisely those bicategories which
are \emph{Cartesian} and satisfy the two axioms (1) \emph{every comonad
has an Eilenberg-Moore object}; and (2) \emph{every left adjoint arrow
is comonadic}\footnote{Note that this characterization considers spans over a category with
both pullbacks and a terminal object (a finitely complete category).}.

In this paper we will give another (very different) characterization
of bicategories of spans. Unlike the aforementioned characterization,
which heavily relies on the adjoint properties of spans (even the
definition of Cartesian bicategory involves adjoints within the bicategory)
ours makes no mention of adjunctions; instead these adjoint properties
of spans are seen as direct consequences of its underlying ``generic
bicategory'' structure.

G\emph{eneric bicategories} were introduced by the author in order
to better understand the universal properties bicategories of spans
of polynomials \cite{WalkerGeneric}, and can be defined as bicategories
$\mathscr{C}$ such that for each 1-cell $c\colon X\to Z$ and object
$Y$ the presheaf
\[
\mathscr{C}_{X,Z}\left(c,-\circ-\right)\colon\mathscr{C}_{Y,Z}\times\mathscr{C}_{X,Y}\to\mathbf{Set}
\]
is a coproduct of representables. For a fixed $X,Y,Z$ and $c$, the
``generic 2-cells out of $c$'' may be defined as the initial objects
within the connected components of the \emph{category of elements
of this presheaf}; that is the category whose objects are 2-cells
from $c$ into a composite of 1-cells through $Y$, written (with
composition in diagrammatic order) $c\Rightarrow a;b$, and whose
morphisms from an object $\alpha\colon c\Rightarrow a;b$ to an object
$\beta\colon c\Rightarrow a';b'$ are pairs of 2-cells $a\Rightarrow a'$
and $b\Rightarrow b'$ which when pasted with $\alpha$ give $\beta$.
A generic 2-cell is then an object $\delta\colon c\Rightarrow l;r$
which is initial in its connected component in this category.

In the case of spans, we have for each span $\left(s,t\right)\colon X\to Z$
with vertex $T$, and each object $Y$, natural isomorphisms\footnote{In general, the indexing set of a multi-adjoint is the set of connected
components of the category of elements of the relevant presheaf. Since
each connected component has a unique equivalence class of initial
objects, this is also the set of equivalence classes of generics.}
\[
\cat{Span}\left(\mathcal{E}\right)_{X,Z}\left[\left(s,t\right),-\circ-\right]\cong\sum_{h\colon T\to Y}\cat{Span}\left(\mathcal{E}\right)_{X,Y}\left(\left(s,h\right),-\right)\cdot\cat{Span}\left(\mathcal{E}\right)_{Y,Z}\left(\left(h,t\right),-\right)
\]
so that $\cat{Span}\left(\mathcal{E}\right)$ is a generic bicategory,
and the generic 2-cells (which are recovered by substituting a pair
of identity 2-cells on the right above) are the morphisms of spans
$\left(s,t\right)\Rightarrow\left(s,h\right);\left(h,t\right)$ given
as diagrams
\begin{equation}
\xymatrix@R=1em{ &  & T\ar@/^{1pc}/[rrddd]^{t}\ar@/_{1pc}/[llddd]_{s}\ar@{..>}[d]^{\delta}\\
 &  & M\ar[rd]^{\pi_{2}}\ar[ld]_{\pi_{1}}\\
 & T\ar[rd]^{h}\ar[ld]_{s} &  & T\ar[rd]^{t}\ar[ld]_{h}\\
X &  & Y &  & Z
}
\label{facofgamma}
\end{equation}
such that $\pi_{1}\delta=\pi_{2}\delta=1_{T}$. We now observe a property
unusual to generic bicategories, we have a ``best generic 2-cell''
(which we will call an \emph{initial generic}), given by taking $h$
to be the identity. So for each span $\left(s,t\right)$, we have
an initial generic $\left(s,t\right)\Rightarrow\left(s,1\right);\left(1,t\right)$,
and a general generic 2-cell $\left(s,t\right)\Rightarrow\left(s,h\right);\left(h,t\right)$
may be recovered by pasting this initial generic with the generic
2-cell $\left(1,1\right)\Rightarrow\left(1,h\right);\left(h,1\right)$.

As any 2-cell $\left(s,t\right)\Rightarrow\left(a,b\right);\left(c,d\right)$
factors through an essentially unique generic 2-cell $\left(s,t\right)\Rightarrow\left(s,h\right);\left(h,t\right)$,
and this itself factors through the initial generic $\left(s,t\right)\Rightarrow\left(s,1\right);\left(1,t\right)$,
we see that bicategories of spans satisfy the existence part of the
following condition.

\begin{ax}[Every 1-cell has an initial generic 2-cell]\label{initialgeneric}
For every 1-cell $c\colon X\to Z$ in $\mathscr{C}$, there exists
an invertible generic 2-cell $\delta\colon c\Rightarrow l;r$ (through
some object $Y$) which is universal in that given any 2-cell $\gamma\colon c\Rightarrow a;b$
(through another object $Y'$) there exists a generic 2-cell $\eta\colon1_{Y}\Rightarrow h;k$
and 2-cells $\alpha$ and $\beta$ as below factoring $\gamma$ as
\begin{equation}
\xymatrix@=1.5em{ &  & \dtwocell{d}{\delta}\\
X\ar@/^{2pc}/[rrrr]^{c}\ar[r]_{l}\ar@/_{1pc}/[rrd]_{a} & Y\ar[rr]_{\textnormal{id}}\ar[rd]_{h} & \dtwocell{d}{\eta}\dtwocell[0.55]{dll}{\alpha}\dtwocell[0.55]{drr}{\beta} & Y\ar[r]_{r} & Z.\\
\; & \; & Y'\ar[ur]_{k}\ar@/_{1pc}/[urr]_{b} & \; & \;
}
\label{initfac}
\end{equation}
Moreover, such a factorization is unique in that given another factorization
\[
\xymatrix@=1.5em{ &  & \dtwocell{d}{\delta}\\
X\ar@/^{2pc}/[rrrr]^{c}\ar[r]_{l}\ar@/_{1pc}/[rrd]_{a} & Y\ar[rr]_{\textnormal{id}}\ar[rd]_{h'} & \dtwocell{d}{\eta'}\dtwocell[0.55]{dll}{\alpha'}\dtwocell[0.55]{drr}{\beta'} & Y\ar[r]_{r} & Z\\
\; & \; & Y'\ar[ur]_{k'}\ar@/_{1pc}/[urr]_{b} & \; & \;
}
\]
with $\eta'\colon1_{Y}\Rightarrow h';k'$ generic, we have induced
comparison isomorphisms $h\cong h'$ and $k\cong k'$ coherent with
$\eta$ and $\eta'$ (necessarily unique by genericity of $\eta$)
and $\alpha,\alpha',\beta,\beta'$. 

\end{ax}

To verify the uniqueness condition mentioned above is satisfied, it
is easiest to first consider the case where $\eta'$ is a representative
generic 2-cell, that is a 2-cell of the form $\eta'\colon1\Rightarrow\left(1,h'\right);\left(h',1\right)$.
In this case we have that $\gamma$ factors through the generic 2-cell
$\left(s,t\right)\Rightarrow\left(s,h\right);\left(h,t\right)$ as
well as the generic 2-cell $\left(s,t\right)\Rightarrow\left(s,h'\right);\left(h',t\right)$,
but since $\cat{Span}\left(\mathcal{E}\right)_{X,Z}\left[\left(s,t\right),\left(a,b\right);\left(c,d\right)\right]$
is isomorphic to 
\[
\sum_{h\colon T\to Y}\cat{Span}\left(\mathcal{E}\right)_{X,Y}\left(\left(s,h\right),\left(a,b\right)\right)\cdot\cat{Span}\left(\mathcal{E}\right)_{Y,Z}\left(\left(h,t\right),\left(c,d\right)\right)
\]
and both the triples $\left(h,\alpha,\beta\right)$ and $\left(h',\alpha',\beta'\right)$
correspond to $\gamma$ under this bijection, we must conclude the
two triples are equal. Thus we have strict uniqueness when one requires
$\eta'$ be a representative generic 2-cell. Given a factorization
of $\gamma$ as on the left below (with $\eta'$ not assumed representative)
\[
\xymatrix@=1.5em{ &  & \dtwocell{d}{\delta} &  &  &  &  &  & \dtwocell[0.4]{d}{\delta}\\
X\ar@/^{2pc}/[rrrr]^{c}\ar[r]_{l}\ar@/_{1pc}/[rrd]_{a} & Y\ar[rr]_{\textnormal{id}}\ar[rd]_{h'} & \dtwocell{d}{\eta'}\dtwocell[0.55]{dll}{\alpha'}\dtwocell[0.55]{drr}{\beta'} & Y\ar[r]_{r} & Z & = & X\ar@/^{2pc}/[rrrr]^{c}\ar[r]^{l}\ar@/_{1pc}/[rrd]_{a} & Y\ar[rr]^{1}\ar@/^{0.6pc}/[rd]|-{m}\ar@/_{0.6pc}/[rd]|-{h'} & \dtwocell[0.3]{d}{\sigma}\dtwocell[0.6]{dll}{\alpha'}\dtwocell[0.55]{drr}{\beta'}\dltwocell{ld}{}\drtwocell{rd}{} & Y\ar[r]^{r} & Z\\
\; & \; & Y'\ar[ur]_{k'}\ar@/_{1pc}/[urr]_{b} & \; & \; &  & \; & \; & Y'\ar@/_{1pc}/[urr]_{b}\ar@/^{0.6pc}/[ru]|-{n}\ar@/_{0.6pc}/[ru]|-{k'} & \; & \;
}
\]
we may factor $\eta'$ through a representative generic $\sigma$
as on the right above (the induced comparisons being of course invertible).
The strict uniqueness that holds with representatives then ensures
the comparisons are coherent with $\alpha,\alpha'$ and $\beta,\beta$.

Finally, there is one more crucial observation we must make. Given
any initial generic pasted with a generic out of an identity $\left(s,t\right)\Rightarrow\left(s,1\right);\left(1,h\right);\left(h,1\right);\left(1,t\right)$
we have that the identities $\left(s,h\right)\Rightarrow\left(s,1\right);\left(1,h\right)$
and $\left(h,t\right)\Rightarrow\left(h,1\right);\left(1,t\right)$
are initial generics, so that bicategories of spans also satisfy the
following condition.

\begin{ax}[Initial factorizations yield initial generics]\label{respinit}
For every such factorization \eqref{initfac} the identities $\left(l;h\right)\Rightarrow l;h$
and $\left(k;r\right)\Rightarrow k;r$ are initial generics.

\end{ax}

The purpose of this paper is to address the question: what \emph{really}
are bicategories of spans? To which we propose the answer that bicategories
of spans are precisely the ``perfect'' generic bicategories; meaning
those generic bicategories for which we have not only generic 2-cells,
but ``best generic 2-cells'' (Axiom \ref{initialgeneric}), and
these best generics induce other best generics (Axiom \ref{respinit}).

\section{Characterizing bicategories of spans}

The goal of this paper is to show that just Axioms \ref{initialgeneric}
and \ref{respinit} are enough to characterize bicategories of spans.
However, at least in the authors opinion, this is not an obvious fact.
Whilst each 1-cell $c$ having an initial generic $c\cong l;r$ is
suggestive of a mapping into spans, there are many points that are
not clear. For example we must deduce that 2-cells are morphisms of
spans, and deduce that composition is by pullback. Moreover, such
facts will rely on one the more fundamental properties of the span
construction used extensively in its characterization in terms of
Cartesian bicategories \cite{spanscartesian}; that any 2-cell between
left adjoints is both unique and invertible, and even this fact is
not clear.

\subsection{Preliminaries}

The purpose of this subsection is to mention some properties of generic
bicategories that will be needed. The properties mentioned here do
not rely on axioms \ref{initialgeneric} or \ref{respinit}.

The following two lemmata are recalled directly from \cite{WalkerGeneric}.
The following recalls that the 3-dimensional version of generic cells
(3-ary generics) are constructed by composing two 2-ary generics.
\begin{lem}
\label{3ary}\cite[Lemma 13]{WalkerGeneric} Suppose $\mathscr{C}$
is a generic bicategory. A 2-cell $\delta\colon c\Rightarrow l;m;r$
is 3-ary generic\footnote{That is initial in a connected component of the category whose objects
are 2-cells $c\Rightarrow x;y;z$ and whose morphisms are coherent
triples of 2-cells $x\Rightarrow x'$, $y\Rightarrow y'$, $z\Rightarrow z'$.} if and only if $\delta$ can be factored as a composite of two 2-ary
generics as on the left below if and only if $\delta$ can be factored
as a composite of two 2-ary generics as on the right below
\[
\xymatrix{c\ar@{=>}[r]^{\sigma_{1}} & l;q\ar@{=>}[r]^{l;\sigma_{2}\quad} & l;m;r &  & c\ar@{=>}[r]^{\sigma_{3}} & p;r\ar@{=>}[r]^{\sigma_{4};r\quad} & l;m;r}
\]
\end{lem}

The following fact will only be used to check that $1_{X}\Rightarrow1_{X};1_{X}$
is always an initial generic; a property we should definitely expect.
\begin{lem}
\label{unitorinvert}\cite[Lemma 11]{WalkerGeneric} Suppose $\mathscr{C}$
is a generic bicategory. If a left unitor $c\Rightarrow1;c$ in $\mathscr{C}$
factors through a generic $c\Rightarrow l;r$ then the induced $r\Rightarrow c$
is invertible.
\end{lem}

The desired property then follows almost immediately.

\begin{cor}
\label{idinit} Suppose $\mathscr{C}$ is a generic bicategory. Then
the unitor $1_{X}\Rightarrow1_{X};1_{X}$ is always an initial generic.
\end{cor}

\begin{proof}
We first factor $1_{X}\Rightarrow1_{X};1_{X}$ through a generic 2-cell.
By Lemma \ref{unitorinvert}, and the fact this both a left and right
unitor, both induced comparison 2-cells are invertible, so that $1_{X}\Rightarrow1_{X};1_{X}$
must itself be generic. It is then completely trivial to directly
verify that $1_{X}\Rightarrow1_{X};1_{X}$ satisfies the universal
property of an initial generic 2-cell.
\end{proof}
Whilst our characterization makes no mention of adjunctions, the proof
of the characterization will rely on them heavily. It will therefore
be prudent to recall the so called ``mates correspondence'' \cite{2caty}.
\begin{prop}
[Mates correspondence] Suppose $f_{1}\dashv g_{1}$ and $f_{2}\dashv g_{2}$
are adjoint pairs in a bicategory. Then for any 1-cells $p$ and $q$
as below, pasting with units and counits of these adjunctions defines
a natural bijection
\[
\xymatrix{A\ar[d]_{f_{1}}\ar[r]^{p}\dtwocell{dr}{\alpha} & B\ar[d]^{f_{2}} &  &  & A\ar[r]^{p}\dtwocell{dr}{\beta} & B\\
C\ar[r]_{q} & D &  &  & C\ar[r]_{q}\ar[u]^{g_{1}} & D\ar[u]_{g_{2}}
}
\]
between 2-cells $\alpha$ and 2-cells $\beta$ as above.
\end{prop}

In particular, we will need the fact that 2-cells of left adjoints
are unique; as one would expect as this is a fundamental property
of the span construction.
\begin{prop}
\label{adjcellsunique} Suppose $\mathscr{C}$ is a generic bicategory.
A 2-cell of left adjoints in $\mathscr{C}$ (and similarly a 2-cell
of right adjoints) is unique.
\end{prop}

\begin{proof}
Recall that in a generic bicategory, identity 1-cells are subterminal
\cite[Prop. 9]{WalkerGeneric}. It then follows immediately from the
mates correspondence that left adjoints and right adjoints are also
subterminal.
\end{proof}

\subsection{Initial generics are a right followed by a left adjoint}

In this paper we do not view adjointness are one of the defining properties
of spans (which is why our characterization makes no mention of adjunctions);
instead the adjoint properties of spans are seen as a consequence
of the underlying generic bicategory structure. The following results
justify this viewpoint.
\begin{lem}
\label{impliesadjoint} Assume $\mathscr{C}$ is a generic bicategory.
Then if a right unitor $\delta\colon l\Rightarrow l;1$ is an initial
generic 2-cell, $l$ must be a right adjoint. Dually, if a left unitor
$\delta\colon r\Rightarrow1;r$ is an initial generic 2-cell, $r$
must be a left adjoint.
\end{lem}

\begin{proof}
Given that the unitor $\delta\colon l\Rightarrow l;1$ is an initial
generic, we have an induced triple into $1;l$ as below
\[
\xymatrix@=1.5em{ &  & \dtwocell{d}{\delta}\\
X\ar@/^{2pc}/[rrrr]^{l}\ar[r]_{l}\ar@/_{1pc}/[rrd]_{\textnormal{id}} & Y\ar[rr]_{\textnormal{id}}\ar[rd]_{h} & \dtwocell{d}{\eta}\dtwocell[0.55]{dll}{\alpha}\dtwocell[0.55]{drr}{\beta} & Y\ar[r]_{\textnormal{id}} & Y\\
\; & \; & X\ar[ur]_{k}\ar@/_{1pc}/[urr]_{l} & \; & \;
}
\]
We now need to check that $\alpha\colon hl\Rightarrow1$ and $\beta h\cdot\eta\colon1\Rightarrow lh$
define the counit and unit of an adjunction $h\dashv l$. As the above
induced triple of cells is necessarily the identity (up to unitors)
we already have one of the triangle identities. 

Let us now make the observation that the pasting of $\delta$ with
$\eta$ as below is a 3-ary generic cell by Lemma \ref{3ary}
\begin{equation}
\xymatrix@=1.5em{ & \dtwocell[0.3]{d}{\delta} & \dtwocell{d}{\eta}\\
X\ar[r]^{l}\ar@/^{3pc}/[rrr]^{l} & Y\ar[r]^{h}\ar@/^{1.5pc}/[rr]^{\textnormal{id}} & X\ar[r]^{k} & Y
}
\label{3ary}
\end{equation}
and that the above pasted with
\[
\xymatrix@=1.5em{X\ar[r]^{l} & Y\ar[rr]^{\textnormal{id}}\ar@/_{0.5pc}/[rrd]_{h} &  & Y\ar[rr]^{h}\dtwocell{dll}{\beta h\cdot\eta}\dtwocell{dr}{\alpha} &  & X\ar@/^{0pc}/[r]^{k}\ar@/_{2pc}/[r]_{l}\dtwocell{dr}{\beta} & Y\\
 & \; &  & X\ar@/_{0.5pc}/[urr]_{\textnormal{id}}\ar[u]_{l} & \; & \; & \;
}
\]
is equal to, by the previous triangle identity, 

\[
\xymatrix@=1.5em{ & \dtwocell[0.3]{d}{\delta} & \dtwocell{d}{\eta}\\
X\ar[r]^{l}\ar@/^{3pc}/[rrr]^{l} & Y\ar[r]^{h}\ar@/^{1.5pc}/[rr]^{\textnormal{id}} & X\ar[r]^{k}\ar@/_{2pc}/[r]_{l}\dtwocell{dr}{\beta} & Y\\
 &  &  & \;
}
\]
Thus the other triangle identity is recovered from 3-ary genericity
of \eqref{3ary}. The proof of the dual property is similar.
\end{proof}
Since we may paste any initial generic $\delta\colon c\Rightarrow l;r$
with the generic $1\Rightarrow1;1$ and then apply Axiom \ref{respinit},
we may deduce the following.
\begin{cor}
\label{partsareadjoints} Assume $\mathscr{C}$ is a generic bicategory
satisfying axioms \ref{initialgeneric} and \ref{respinit}. Then
for any initial generic $\delta\colon c\Rightarrow l;r$ we have that
$l$ is a right adjoint and $r$ is a left adjoint.
\end{cor}

\subsection{Generics out of identities are precisely units of adjunctions}

In a bicategory of spans the (representative) generics out an identity
are of the form $\left(1,1\right)\Rightarrow\left(1,h\right);\left(h,1\right)$
and coincide with the units of the adjunctions. Here we show that
this fact about bicategories of spans follows from Axioms \ref{initialgeneric}
and \ref{respinit}.
\begin{lem}
\label{genericsareunits} Assume $\mathscr{C}$ is a generic bicategory
satisfying axioms \ref{initialgeneric} and \ref{respinit}. Then
every generic 2-cell out of an identity $\delta\colon1\Rightarrow h;k$
is the unit of an adjunction $h\dashv k$.
\end{lem}

\begin{proof}
By Corollary \ref{idinit} and Axiom \ref{respinit} we that know
both $1;h$ and $k;1$ are initial, and so $h$ and $k$ are left
and right adjoints respectively by Lemma \ref{impliesadjoint}. Thus
the 2-cell $\delta\colon1\Rightarrow kh$ has a mate $h^{*}\Rightarrow k$
where $h^{*}$ is a right adjoint to $h$. However, as
\[
\xymatrix@=1.5em{Y\ar[rd]_{h}\ar[rr]^{1} & \dtwocell[0.4]{d}{\eta}\drtwocell[1]{dr}{} & Y\\
 & Y'\ar[ur]_{h^{*}}\ar@/_{2.3pc}/[ru]_{k} & \;
}
\]
is equal to $\delta$, both $\delta$ and $\eta$ lie in the same
connected component. Hence $\eta$ factors through $\delta$ (which
is initial in its connected component by definition), yielding comparisons
$h\Rightarrow h$ and $k\Rightarrow h^{*}$. By Proposition \ref{adjcellsunique},
we must then conclude the 2-cell $h^{*}\Rightarrow k$ is an isomorphism.
\end{proof}
We now use the above lemma to prove its converse.
\begin{lem}
\label{unitimpliesgen} Assume $\mathscr{C}$ is a generic bicategory
satisfying axioms \ref{initialgeneric} and \ref{respinit}. Then
every unit of an adjunction $\nu\colon1\Rightarrow f;g$ is a generic
2-cell out of an identity.
\end{lem}

\begin{proof}
We again recall Corollary \ref{idinit}, and then factor a unit $\nu\colon1\Rightarrow f;g$
as
\[
\xymatrix@=1.5em{ &  & \dtwocell{d}{\delta}\\
X\ar@/^{2pc}/[rrrr]^{1}\ar[r]_{1}\ar@/_{1pc}/[rrd]_{f} & X\ar[rr]_{1}\ar[rd]_{h} & \dtwocell{d}{\eta}\dtwocell[0.55]{dll}{\alpha}\dtwocell[0.55]{drr}{\beta} & X\ar[r]_{1} & X.\\
\; & \; & Y\ar[ur]_{k}\ar@/_{1pc}/[urr]_{g} & \; & \;
}
\]
Now, noting by Lemma \ref{genericsareunits} that $h\dashv k$, we
have the mates of $\alpha$ and $\beta$ given by $\alpha^{*}\colon g\Rightarrow k$
and $\beta_{*}\colon f\Rightarrow h$, and these are inverse to $\beta$
and $\alpha$ respectively by Proposition \ref{adjcellsunique}.
\end{proof}

\subsection{Adjoint composites are initial}

Whilst we know by Corollary \ref{partsareadjoints} that initial generics
$c\cong l;r$ yield factorizations of any 1-cell $c$ as a right adjoint
followed by a left adjoint, thus giving a way of mapping 1-cells into
spans, we would like this assignation to be essentially surjective.
The following lemma ensures this property.
\begin{lem}
\label{compinit} Assume $\mathscr{C}$ is a generic bicategory satisfying
axioms \ref{initialgeneric} and \ref{respinit}. Then for any composable
right adjoint $s^{*}$ and left adjoint $t$, the composite $s^{*};t$
is initial.
\end{lem}

\begin{proof}
We first note that for any left adjoint $t$, the composite $1;t$
is initial, which is seen by recalling Corollary \ref{idinit} and
applying Axiom \ref{respinit} to the diagram
\[
\xymatrix@=1.5em{ & T\ar[r]^{1} & T\ar[rr]^{1}\ar[rd]_{t} & \dtwocell[0.4]{d}{\eta} & T\ar[r]^{1} & T.\\
\; & \; & \; & Y\ar[ur]_{t^{*}} & \; & \;
}
\]
By then applying Axiom \ref{respinit} to the diagram
\[
\xymatrix@=1.5em{ & T\ar[r]^{1} & T\ar[rr]^{1}\ar[rd]_{s} & \dtwocell[0.4]{d}{\eta'} & T\ar[r]^{t} & Y\\
\; & \; & \; & X\ar[ur]_{s^{*}} & \; & \;
}
\]
we see that $s^{*};t$ is initial.
\end{proof}

\subsection{The left adjoints form a 1-category}

Given a bicategory of spans $\cat{Span}\left(\mathcal{E}\right)$,
one may recover $\mathcal{E}$ as its category of left adjoints. However,
the left adjoints only form a 1-category since any 2-cell between
left adjoints is both unique and invertible. We have already partially
verified this property, showing uniqueness of such 2-cells in Proposition
\ref{adjcellsunique}, thus it just remains to check they are invertible.
\begin{lem}
\label{adjointcellsinvert} Assume $\mathscr{C}$ is a generic bicategory
satisfying axioms \ref{initialgeneric} and \ref{respinit}. Then
any 2-cell $\alpha\colon f_{1}\Rightarrow f_{2}$ between left adjoints
$f_{1}$ and $f_{2}$ is invertible.
\end{lem}

\begin{proof}
Suppose $g_{1}$ and $g_{2}$ are respective right adjoints to $f_{1}$
and $f_{2}$. We then have an equality\vspace*{-0.2cm}
\[
\xymatrix@=1.5em{X\ar[r]^{g_{2}}\ar@/_{1pc}/[rrd]_{1} & Y\ar[rr]^{1}\ar@/^{0.7pc}/[rd]|-{f_{1}}\ar@/_{0.7pc}/[rd]|-{f_{2}} & \dtwocell{d}{\eta_{1}}\dtwocell[0.6]{dll}{\epsilon_{2}}\dtwocell[0.5]{drr}{\epsilon_{1}}\dltwocell{ld}{\alpha} & Y\ar[r]^{f_{1}} & X & \ar@{}[rd]|-{=} &  & X\ar[r]^{g_{2}}\ar@/_{1pc}/[rrd]_{1} & Y\ar[rr]^{1}\ar[rd]_{f_{2}} & \dtwocell{d}{\eta_{2}}\dtwocell[0.6]{dll}{\epsilon_{2}}\dtwocell[0.5]{drr}{\epsilon_{2}} & Y\ar@/^{0.5pc}/[r]^{f_{1}}\ar@/_{0.5pc}/[r]_{f_{2}}\dtwocell{r}{\alpha} & X\\
\; & \; & X\ar[ur]_{g_{1}}\ar@/_{1pc}/[urr]_{1} & \; & \; &  & \; & \; & \; & X\ar[ur]_{g_{2}}\ar@/_{1pc}/[urr]_{1} & \; & \;
}
\]
yielding, as $g_{2};f_{1}$ is initial by Lemma \ref{compinit}, an
isomorphism $f_{1}\cong f_{2}$ from uniqueness of such factorizations,
thus showing $\alpha$ is invertible by Proposition \ref{adjcellsunique}. 
\end{proof}
Since any 2-cell between left adjoints is both unique by Proposition
\ref{adjcellsunique} and invertible by Lemma \ref{adjointcellsinvert}
we have the following.
\begin{cor}
\label{1caty} Assume $\mathscr{C}$ is a generic bicategory satisfying
axioms \ref{initialgeneric} and \ref{respinit}. Then the sub-bicategory
of left adjoint 1-cells in $\mathscr{C}$ is equivalent to a locally
discrete 2-category $\mathcal{E}$ whose morphisms are equivalence
classes of left adjoints.
\end{cor}

\subsection{Initial generics paste with generics to give generics}

The following shows that any initial generic $c\Rightarrow l;r$ pasted
with a generic out an identity $1\Rightarrow h;k$ yields a generic
$c\Rightarrow\left(l;h\right);\left(k;r\right)$. With this proven
it will then follow that when a 2-cell $\gamma\colon c\Rightarrow a;b$
is factored through an initial generic as in \eqref{initfac}, the
induced comparisons $\alpha$ and $\beta$ must be invertible whenever
$\gamma$ is generic.
\begin{lem}
\label{compgen} Suppose $\mathscr{C}$ is a generic bicategory, and
that a 1-cell $c$ admits an initial generic $c\cong l;r$ through
an object $Y$. Then for any generic out an identity $1_{Y}\Rightarrow h;k$
the pasting $c\Rightarrow\left(l;h\right);\left(k;r\right)$ is a
2-ary generic.
\end{lem}

\begin{proof}
Consider a 1-cell $c$ with initial generic $l;r$ pasted with a generic
2-cell $\eta\colon1\Rightarrow h;k$ as on the left below. There is
no loss in generality assuming $\eta$ is a representative generic.
\[
\xymatrix@=1.5em{ &  & \dtwocell{d}{\delta} &  &  &  &  &  &  & \dtwocell{d}{\xi}\\
X\ar@/^{2pc}/[rrrr]^{c}\ar[r]_{l} & Y\ar[rr]_{\textnormal{id}}\ar[rd]_{h} & \dtwocell{d}{\eta} & Y\ar[r]_{r} & Z &  &  & X\ar[rr]_{L}\ar@/^{2pc}/[rrrr]^{c}\ar@/_{2pc}/[rr]_{l;h} & \dtwocell{d}{\theta} & Y'\ar[rr]_{R}\ar@/_{2pc}/[rr]_{k;r} & \dtwocell{d}{\phi} & Z\\
\; & \; & Y'\ar[ur]_{k} & \; & \; &  &  & \; & \; &  & \; & \;
}
\]
We check that $c\Rightarrow\left(l;h\right);\left(k;r\right)$ is
generic. To see this, we first factor this pasting through a generic
2-cell $\xi$ as on the right above. It remains to check the induced
comparison maps $\theta$ and $\phi$ are invertible.

We proceed by factoring $\xi$ through the initial generic as
\[
\xymatrix@=1.5em{ &  &  &  & \dtwocell{d}{\delta}\\
\xi & = & X\ar@/^{2pc}/[rrrr]^{c}\ar[r]_{l}\ar@/_{1pc}/[rrd]_{L} & Y\ar[rr]_{\textnormal{id}}\ar[rd]_{h'} & \dtwocell{d}{\eta'}\dtwocell[0.55]{dll}{\alpha}\dtwocell[0.55]{drr}{\beta} & Y\ar[r]_{r} & Z\\
 &  & \; & \; & Y'\ar[ur]_{k'}\ar@/_{1pc}/[urr]_{R} & \; & \;
}
\]
where we chose $\eta'$ to be a representative generic. This gives
the equality
\[
\xymatrix@=1.5em{ &  & \dtwocell{d}{\delta} &  &  &  &  &  & \dtwocell{d}{\delta}\\
X\ar@/^{2pc}/[rrrr]^{c}\ar[r]_{l}\ar@/_{1pc}/[rrd]_{L}\ar@/_{3.5pc}/[rrd]_{l;h} & Y\ar[rr]_{\textnormal{id}}\ar[rd]_{h'} & \dtwocell{d}{\eta'}\dtwocell[0.55]{dll}{\alpha}\dtwocell[0.55]{drr}{\beta} & Y\ar[r]_{r} & Z &  & X\ar@/^{2pc}/[rrrr]^{c}\ar[r]_{l} & Y\ar[rr]_{\textnormal{id}}\ar[rd]_{h} & \dtwocell{d}{\eta} & Y\ar[r]_{r} & Z\\
\; & \;\dtwocell{d}{\theta} & Y'\ar[ur]_{k'}\ar@/_{1pc}/[urr]_{R}\ar@/_{3.5pc}/[rru]_{k;r} & \;\dtwocell{d}{\phi} & \; & = & \; & \; & Y'\ar[ur]_{k} & \; & \;\\
 & \; &  & \;
}
\]
so that by uniqueness $\left(h,k,\eta\right)$ and $\left(h',k',\eta'\right)$
are equal (using that Axiom \ref{initialgeneric} gives strict uniqueness
when one requires $\eta$ and $\eta'$ to be representative) and $\theta\alpha$
and $\varphi\beta$ are identities. Pasting both sides above with
$\alpha$ and $\beta$ gives the equality
\[
\xymatrix@=1.5em{ &  & \dtwocell{d}{\xi} &  &  &  &  &  & \dtwocell{d}{\xi}\\
X\ar[rr]_{L}\ar@/^{2pc}/[rrrr]^{c}\ar@/_{2pc}/[rr]_{L} & \dtwocell[0.55]{d}{\alpha\theta} & Y'\ar[rr]_{R}\ar@/_{2pc}/[rr]_{R} & \dtwocell[0.55]{d}{\beta\phi} & Z & = & X\ar[rr]_{L}\ar@/^{2pc}/[rrrr]^{c} &  & Y'\ar[rr]_{R} &  & Z\\
\; & \; &  & \; & \; &  & \; & \; &  & \; & \;
}
\]
so that $\alpha\theta$ and $\beta\phi$ are also identities by genericity
of $\xi$.
\end{proof}

\subsection{Generics are indexed by classes of left adjoint 1-cells}

One of the more interesting properties of a generic bicategory $\mathscr{C}$
is that the 2-cells in $\mathscr{C}$ determine composition in $\mathscr{C}$.
More explicitly, one can define for each triple of objects $X,Y,Z$
and 1-cell $c\colon X\to Z$ in $\mathscr{C}$, the set $\mathfrak{M}_{c}^{X,Y,Z}$
of equivalence classes of generic 2-cells $c\Rightarrow l;r$ factoring
through $Y$. It is then not hard to see that this defines a presheaf
$\mathfrak{M}_{\left(-\right)}^{X,Y,Z}\colon\mathscr{C}_{X,Z}\to\mathbf{Set}$.
Moreover, given an element of this presheaf (that is a generic 2-cell
$\delta\colon c\Rightarrow l;r$) we have the canonical projections
$\delta\mapsto\left(l,r\right)$ and $\delta\mapsto c$ yielding functors
as below
\[
\xymatrix@=1.5em{\mathscr{C}_{X,Y}\times\mathscr{C}_{Y,Z}\ar@{..>}[rr] & \ltwocell[0.4]{d}{} & \mathscr{C}_{X,Z}.\\
 & \el\mathfrak{M}_{\left(-\right)}^{X,Y,Z}\ar[ur]\ar[ul]
}
\]
It then a fact about generic bicategories that the composition functor
defines an absolute left extension above\footnote{Actually, this is an instance of a more general fact which is true
for any familial functor.}. Thus if we have a description of the generic 2-cells, or at least
a way of keeping track of them, we may deduce the composition operation
within our bicategory. It will therefore be useful to find a way of
indexing our generic 2-cells, in the hope of later using this indexing
to show composition in $\mathscr{C}$ is by pullback.
\begin{lem}
\label{genericsarewhiskers} Assume $\mathscr{C}$ is a generic bicategory
satisfying axioms $\ref{initialgeneric}$ and \ref{respinit}. Suppose
$c$ is a 1-cell with initial generic $\delta\colon c\Rightarrow l;r$.
Then a 2-cell $c\Rightarrow a;b$ in $\mathscr{C}$ is generic, if
and only if it is (isomorphic to) a whiskering of a unit $\eta$ of
an adjunction by $l$ and $r$.
\end{lem}

\begin{proof}
$\left(\Rightarrow\right)\colon$ It is clear that any generic 2-cell
$c\Rightarrow a;b$ is a whiskering of a unit by $l$ and $r$, since
we may factor such a 2-cell through the initial generic as
\[
\xymatrix@=1.5em{ &  & \dtwocell{d}{\delta}\\
X\ar@/^{2pc}/[rrrr]^{c}\ar[r]_{l}\ar@/_{1pc}/[rrd]_{a} & Y\ar[rr]_{\textnormal{id}}\ar[rd]_{h} & \dtwocell{d}{\eta}\dtwocell[0.55]{dll}{\alpha}\dtwocell[0.55]{drr}{\beta} & Y\ar[r]_{r} & Z\\
\; & \; & Y'\ar[ur]_{k}\ar@/_{1pc}/[urr]_{b} & \; & \;
}
\]
and note the induced $\alpha$ and $\beta$ are invertible, as $c\Rightarrow\left(l;h\right);\left(k;r\right)$
is generic by Lemma \ref{compgen}.

$\left(\Leftarrow\right)\colon$ Suppose now we are given a unit $\eta$
whiskered by $l$ and $r$ giving a diagram as below
\[
\xymatrix@=1.5em{ &  & \dtwocell{d}{\delta}\\
X\ar[r]_{l}\ar@/^{2pc}/[rrrr]^{c} & Y\ar[rr]_{\textnormal{id}}\ar[rd]_{h} & \dtwocell{d}{\eta} & Y\ar[r]_{r} & Z\\
\; & \; & Y'\ar[ur]_{k} & \; & \;
}
\]
We factor the above through this initial generic, yielding the same
diagram. We then note that $\left(l;r\right)\Rightarrow\left(l;h\right);\left(k;r\right)$
is generic by Lemma \ref{compgen}.
\end{proof}
We can now apply the above lemma to give a simple description of our
indexing sets $\mathfrak{M}_{c}^{X,Y',Z}$ of generic 2-cells.
\begin{lem}
\label{genericsindexh} Assume $\mathscr{C}$ is a generic bicategory
satisfying axioms \ref{initialgeneric} and \ref{respinit}. Then
for any 1-cell $c\colon X\to Z$ (with a chosen initial generic $l;r$
through an object $Y$) and object $Y'$, the set of equivalence classes
of generic 2-cells out of $c$ and through $Y'$, denoted $\mathfrak{M}_{c}^{X,Y',Z}$,
is the set of equivalence classes of left adjoints $h\colon Y\to Y'$.
\end{lem}

\begin{proof}
Suppose we are given a 1-cell $c\colon X\to Z$ with a chosen initial
generic $l;r$ through an object $Y$. We regard $l$ and $r$ as
being fixed.

We see that from any generic 2-cell out of $c$, we recover by Lemma
\ref{genericsarewhiskers}, a 1-cell $h\colon Y\to Y'$, a right adjoint
$k\colon Y'\to Y$ and a unit $\eta\colon1\Rightarrow kh$. Conversely,
given such a triple $\left(h,k,\eta\right)$ we recover a generic
2-cell out of $c$ by whiskering.

Let us suppose we have two isomorphic generic 2-cells out of $c$,
then one is isomorphic to a whiskering of data $\left(h,k,\eta\right)$
and the other a whiskering of $\left(h',k',\eta'\right)$, and thus
we have isomorphisms $\alpha$ and $\beta$ yielding an equality as
below
\[
\xymatrix@=1.5em{X\ar[r]^{l}\ar@/_{1pc}/[rrd]_{l;h'} & Y\ar[rr]^{1}\ar[rd]_{h} & \dtwocell[0.45]{d}{\eta}\dtwocell[0.55]{dll}{\alpha}\dtwocell[0.55]{drr}{\beta} & Y\ar[r]^{r} & Z & \ar@{}[rd]|-{=} &  & X\ar[r]^{l} & Y\ar[rr]^{1}\ar[rd]_{h'} & \dtwocell[0.45]{d}{\eta'} & Y\ar[r]^{r} & Z\\
\; & \; & Y'\ar[ur]_{k}\ar@/_{1pc}/[urr]_{k';r} & \; & \; &  & \; & \; & \; & Y'\ar[ur]_{k'} & \; & \;
}
\]
By uniqueness, this induces comparison isomorphisms $h\cong h'$ and
$k\cong k'$ compatible with $\eta$. Indeed, denoting the induced
isomorphism $h\Rightarrow h'$ by $\theta$ (and forgetting the induced
$k\Rightarrow k'$) we may take $\xi\colon k'\Rightarrow k$ to be
the composite
\[
\xymatrix{k'\myar{\eta k'}{r} & khk'\myar{k\theta k'}{r} & kh'k'\myar{k\epsilon'}{r} & k}
\]
and one can verify that $\eta'$ pasted with $\theta^{-1}$ and $\xi$
is $\eta$, so that from an isomorphism $h\Rightarrow h'$ we recover
an isomorphism of generic 2-cells $\left(h,k,\eta\right)\cong\left(h',k',\eta'\right)$.
It follows the assignment sending a generic 2-cell to its representative
triple then to its left adjoint component
\[
\textnormal{generic 2-cell}\mapsto\left(h,k,\eta\right)\mapsto h
\]
reflects and preserves isomorphisms. Thus we are to identify two generic
2-cells precisely when the left adjoint components are isomorphic
(that is equal in their equivalence class).
\end{proof}
\begin{rem}
Note that a representative generic corresponding to a $h\colon Y\to Y'$
has the form $c\Rightarrow\left(l;h\right);\left(k;r\right)$, so
that the left and right projections are $l;h$ and $k;r$ respectively.
This then defines the projections from $\mathfrak{M}_{c}^{X,Y',Z}$
to $\mathscr{C}_{X,Y'}$ and $\mathscr{C}_{Y',Z}$ .
\end{rem}

\subsection{The main theorem}

We now have all of the necessary ingredients to prove our characterization
of spans. We will first show that for a generic $\mathscr{C}$ satisfying
axioms \ref{initialgeneric} and \ref{respinit} the hom-categories
must be equivalent to hom-categories of spans, and then proceed by
using the generic bicategory structure (and our indexing of the generics)
to deduce that composition must be given by pullback.
\begin{thm}
The following are equivalent for any given bicategory $\mathscr{C}$:
\begin{enumerate}
\item $\mathscr{C}$ is equivalent to a bicategory $\cat{Span}\left(\mathcal{E}\right)$
for a category $\mathcal{E}$ with pullbacks;
\item $\mathscr{C}$ is generic and satisfies axioms \ref{initialgeneric}
and \ref{respinit}.
\end{enumerate}
\end{thm}

\begin{proof}
The implication $\left(1\right)\Rightarrow\left(2\right)$ was shown
in the introduction, so it remains to check $\left(2\right)\Rightarrow\left(1\right)$.
Let $X$ and $Z$ be given objects. We take $\mathcal{E}$ to be the
1-category of representative left adjoints given by Corollary \ref{1caty},
and define the functor $\mathscr{C}_{X,Z}\to\cat{Span}\left(\mathcal{E}\right)_{X,Z}$
by the assignment sending a 1-cell $c$ with chosen initial generic
$c\Rightarrow l;r$ to the span
\[
\xymatrix@R=1em{ & Y\ar[rd]^{r}\ar[ld]_{l_{*}}\\
X &  & Z
}
\]
where $l_{*}$ is a representative left adjoint of $l$. To assign
a 2-cell $\alpha\colon c\Rightarrow c'$ to a morphism of spans, we
first factor $\alpha$ through the initial generic $c\Rightarrow l;r$
giving a diagram
\begin{equation}
\xymatrix{X\ar[r]^{l}\ar@{=}[d]\dtwocell{rd}{\theta} & Y\ar[r]^{1}\ar[d]_{h}\dtwocell{rd}{\eta} & Y\ar[r]^{r}\dtwocell{rd}{\phi} & Z\ar@{=}[d]\\
X\ar[r]_{l'} & Y'\ar[r]_{1} & Y'\ar[r]_{r'}\ar[u]_{k} & Z
}
\label{3pasting}
\end{equation}
We then note that we have (from pasting $\theta$ with $\eta$, and
$\eta$ with $\phi$ respectively) 2-cells $l\Rightarrow l';k$ and
$r\Rightarrow h;r'$. As morphisms of left adjoints correspond to
morphisms of right adjoints in the opposite direction by the mates
correspondence, we also have 2-cells $h;l'_{*}\Rightarrow l_{*}$
and $r'^{*};k\Rightarrow r^{*}$ respectively. In particular, as the
2-cells of left adjoints $r\Rightarrow h;r'$ and $h;l'_{*}\Rightarrow l_{*}$
collapse to identities in $\mathcal{E}$, we have a commuting diagram
\begin{equation}
\xymatrix@R=1em{ & Y\ar[rd]^{r}\ar[ld]_{l_{*}}\ar[dd]|-{h}\\
X &  & Z.\\
 & Y'\ar[ru]_{r'}\ar[lu]^{l'_{*}}
}
\label{morspans}
\end{equation}

Moreover, this assignment is fully faithful since given a morphism
of spans as in \eqref{morspans}, we may recover a diagram \eqref{3pasting}
by taking $\varphi$ as the mate of $r\cong h;r'$ under the adjunction
$h\dashv k$, and $\theta$ as the mate of $h;l'_{*}\cong l_{*}$
under the adjunctions $l_{*}\dashv l$ and $l'_{*}\dashv l'$. It
is clear this application of the mates correspondence defines a bijection.

To see that this assignment is essentially surjective, suppose we
are given a span
\[
\xymatrix@R=1em{ & T\ar[rd]^{t}\ar[ld]_{s}\\
X &  & Z
}
\]
and note that we have the initial generic $\left(s^{*};t\right)\Rightarrow s^{*};t$
by Lemma \ref{compinit}. By universality of the chosen initial generic
$l;r$ we have an induced comparison into $s^{*};t$ as on the top
half below
\begin{equation}
\xymatrix{X\ar[r]^{l}\ar@{=}[d]\dtwocell{rd}{\theta_{1}} & Y\ar[r]^{1}\ar[d]^{h_{1}}\dtwocell{rd}{\eta_{1}} & Y\ar[r]^{r}\dtwocell{rd}{\phi_{1}} & Z\ar@{=}[d]\\
X\ar[r]_{s^{*}}\dtwocell{rd}{\theta_{2}}\ar@{=}[d] & T\ar[r]_{1}\dtwocell{rd}{\eta_{2}}\ar[d]^{h_{2}} & T\ar[r]_{t}\ar[u]_{k_{1}}\dtwocell{rd}{\phi_{2}} & Z\ar@{=}[d]\\
X\ar[r]_{l} & Y\ar[r]_{1} & Y\ar[r]_{r}\ar[u]_{k_{2}} & Z
}
\label{6paste}
\end{equation}
and conversely by universality of $s^{*};t$ we have an induced comparison
into $l;r$ as on the bottom half below. By uniqueness (and the fact
$1_{Y}\Rightarrow1_{Y};1_{Y}$ is a generic 2-cell) it follows that
$h_{1};h_{2}\cong1_{Y}$ in $\mathscr{C}$, meaning that $h_{1};h_{2}=1_{Y}$
in $\mathcal{E}$. Clearly, we also have the dual so that $h_{2};h_{1}=1_{T}$,
and thus applying our functor to the top half of \eqref{6paste} yields
an isomorphism of spans $\left(l,r\right)\cong\left(s,t\right)$,
proving essential surjectivity.

We now note that the bicategory structure on $\mathscr{C}$ transports
through the equivalences $\mathscr{C}_{X,Z}\to\cat{Span}\left(\mathcal{E}\right)_{X,Z}$
by Doctrinal adjunction \cite{doctrinal} (viewing bicategories as
the objects of the 2-category of bicategories, pseudofunctors, and
icons \cite{icons}), and thus the family of hom-categories $\cat{Span}\left(\mathcal{E}\right)_{X,Y}$
admits the structure of a bicategory with composition given by
\[
\xymatrix@R=0.5em{\cat{Span}\left(\mathcal{E}\right)_{X,Y}\times\cat{Span}\left(\mathcal{E}\right)_{Y,Z}\ar[r] & \mathscr{C}_{X,Y}\times\mathscr{C}_{Y,Z}\myar{\circ}{r} & \mathscr{C}_{X,Z}\ar[r] & \cat{Span}\left(\mathcal{E}\right)_{X,Z}\\
\left(a,b\right),\left(c,d\right)\ar@{|->}[r] & \left(a^{*};b\right),\left(c^{*};d\right)\ar@{|->}[r] & l^{*};r\ar@{|->}[r] & \left(l,r\right)
}
\]
where $l^{*};r$ is a chosen initial generic for $\left(a^{*};b\right);\left(c^{*};d\right)$,
and $l$ is a representative for the left adjoints to $l^{*}$. We
know composition in $\mathscr{C}$ satisfies, for any initial generic
$s^{*};t\colon X\to T\to Z$ and initial generics $a^{*};b\colon X\to Y$
and $c^{*};d\colon Y\to Z$, the natural bijection
\[
\mathscr{C}_{X,Z}\left[\left(s^{*};t\right),\left(a^{*};b\right);\left(c^{*};d\right)\right]\cong\sum_{m\in\mathfrak{M}_{c}^{X,Y,Z}}\mathscr{C}_{X,Y}\left(l_{m},\left(a^{*};b\right)\right)\times\mathscr{C}_{X,Y}\left(r_{m},\left(c^{*};d\right)\right)
\]
so that by Lemma \ref{genericsindexh}
\[
\mathscr{C}_{X,Z}\left[\left(s^{*};t\right),\left(l^{*};r\right)\right]\cong\sum_{h\colon T\to Y}\mathscr{C}_{X,Y}\left(\left(s^{*};h\right),\left(a^{*};b\right)\right)\times\mathscr{C}_{X,Y}\left(\left(h^{*};t\right),\left(c^{*};d\right)\right)
\]
which says for all $s$ and $t$ in $\mathcal{E}$, giving a morphism
of spans (that is a morphism into the vertex of $\left(l,r\right)$
as on the left below) is the same as giving a cone
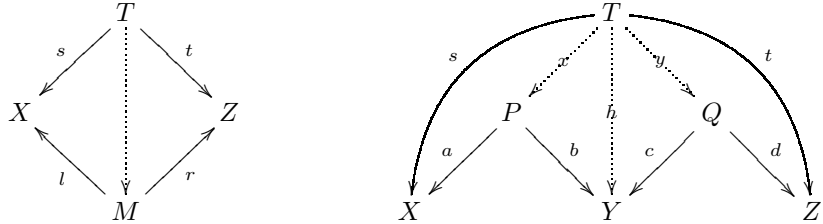
\[
\xymatrix{ & T\ar[rd]^{t}\ar[ld]_{s}\ar@{..>}[dd] &  &  &  &  & T\ar@{..>}[dd]|-{h}\ar@{..>}[dr]|-{y}\ar@{..>}[dl]|-{x}\ar@/^{2pc}/[rrdd]^{t}\ar@/_{2pc}/[lldd]_{s}\\
X &  & Z &  &  & P\ar[rd]^{b}\ar[ld]_{a} &  & Q\ar[rd]^{d}\ar[ld]_{c}\\
 & M\ar[ru]_{r}\ar[lu]^{l} &  &  & X &  & Y &  & Z
}
\]
as on the right above, thus showing that composition is by pullback
(after taking $\left(s,t\right)$ to be $\left(l,r\right)$ in order
to recover the limiting cone).
\end{proof}
\bibliographystyle{siam}
\bibliography{referencesspansgeneric}

\end{document}